\documentclass[12pt,twoside]{article}
\usepackage{amsmath,amssymb}
\usepackage{amsthm}
\usepackage{upref}
\usepackage{graphicx,amssymb}
\usepackage{color}
\numberwithin{equation}{section}

\theoremstyle{plain}
\newtheorem{theorem}{Theorem}[section]
\newtheorem{proposition}[theorem]{Proposition}

\theoremstyle{definition}
\newtheorem{remark}[theorem]{Remark}
\newtheorem{problem}[theorem]{Problem}

\begin{document}

\newcommand{\eq}{equation}
\newcommand{\real}{\ensuremath{\mathbb R}}
\newcommand{\comp}{\ensuremath{\mathbb C}}
\newcommand{\rn}{\ensuremath{{\mathbb R}^n}}
\newcommand{\tn}{\ensuremath{{\mathbb T}^n}}
\newcommand{\rnp}{\ensuremath{\real^n_+}}
\newcommand{\rnn}{\ensuremath{\real^n_-}}
\newcommand{\Rn}{\ensuremath{{\mathbb R}^{n-1}}}
\newcommand{\Zn}{\ensuremath{{\mathbb Z}^{n-1}}}
\newcommand{\no}{\ensuremath{\nat_0}}
\newcommand{\ganz}{\ensuremath{\mathbb Z}}
\newcommand{\zn}{\ensuremath{{\mathbb Z}^n}}
\newcommand{\zom}{\ensuremath{{\mathbb Z}_{\Om}}}
\newcommand{\zOm}{\ensuremath{{\mathbb Z}^{\Om}}}
\newcommand{\As}{\ensuremath{A^s_{p,q}}}
\newcommand{\Bs}{\ensuremath{B^s_{p,q}}}
\newcommand{\Fs}{\ensuremath{F^s_{p,q}}}
\newcommand{\Fsr}{\ensuremath{F^{s,\rloc}_{p,q}}}
\newcommand{\nat}{\ensuremath{\mathbb N}}
\newcommand{\Om}{\ensuremath{\Omega}}
\newcommand{\di}{\ensuremath{{\mathrm d}}}
\newcommand{\sn}{\ensuremath{{\mathbb S}^{n-1}}}
\newcommand{\Ac}{\ensuremath{\mathcal A}}
\newcommand{\Acs}{\ensuremath{\Ac^s_{p,q}}}
\newcommand{\Bc}{\ensuremath{\mathcal B}}
\newcommand{\Cc}{\ensuremath{\mathcal C}}
\newcommand{\cc}{{\scriptsize $\Cc$}${}^s (\rn)$}
\newcommand{\ccd}{{\scriptsize $\Cc$}${}^s (\rn, \delta)$}
\newcommand{\Fc}{\ensuremath{\mathcal F}}
\newcommand{\Lc}{\ensuremath{\mathcal L}}
\newcommand{\Mc}{\ensuremath{\mathcal M}}
\newcommand{\Ec}{\ensuremath{\mathcal E}}
\newcommand{\Pc}{\ensuremath{\mathcal P}}
\newcommand{\Efr}{\ensuremath{\mathfrak E}}
\newcommand{\Mfr}{\ensuremath{\mathfrak M}}
\newcommand{\Mbf}{\ensuremath{\mathbf M}}
\newcommand{\Dbb}{\ensuremath{\mathbb D}}
\newcommand{\Lbb}{\ensuremath{\mathbb L}}
\newcommand{\Pbb}{\ensuremath{\mathbb P}}
\newcommand{\Qbb}{\ensuremath{\mathbb Q}}
\newcommand{\Rbb}{\ensuremath{\mathbb R}}
\newcommand{\vp}{\ensuremath{\varphi}}
\newcommand{\hra}{\ensuremath{\hookrightarrow}}
\newcommand{\supp}{\ensuremath{\mathrm{supp \,}}}
\newcommand{\ssupp}{\ensuremath{\mathrm{sing \ supp\,}}}
\newcommand{\dist}{\ensuremath{\mathrm{dist \,}}}
\newcommand{\unif}{\ensuremath{\mathrm{unif}}}
\newcommand{\ve}{\ensuremath{\varepsilon}}
\newcommand{\vk}{\ensuremath{\varkappa}}
\newcommand{\vr}{\ensuremath{\varrho}}
\newcommand{\pa}{\ensuremath{\partial}}
\newcommand{\oa}{\ensuremath{\overline{a}}}
\newcommand{\ob}{\ensuremath{\overline{b}}}
\newcommand{\of}{\ensuremath{\overline{f}}}
\newcommand{\LA}{\ensuremath{L^r\!\As}}
\newcommand{\LcA}{\ensuremath{\Lc^{r}\!A^s_{p,q}}}
\newcommand{\LcdA}{\ensuremath{\Lc^{r}\!A^{s+d}_{p,q}}}
\newcommand{\LcB}{\ensuremath{\Lc^{r}\!B^s_{p,q}}}
\newcommand{\LcF}{\ensuremath{\Lc^{r}\!F^s_{p,q}}}
\newcommand{\Lf}{\ensuremath{L^r\!f^s_{p,q}}}
\newcommand{\La}{\ensuremath{\Lambda}}
\newcommand{\Lob}{\ensuremath{L^r \ob{}^s_{p,q}}}
\newcommand{\Lof}{\ensuremath{L^r \of{}^s_{p,q}}}
\newcommand{\Loa}{\ensuremath{L^r\, \oa{}^s_{p,q}}}
\newcommand{\Lcoa}{\ensuremath{\Lc^{r}\oa{}^s_{p,q}}}
\newcommand{\Lcob}{\ensuremath{\Lc^{r}\ob{}^s_{p,q}}}
\newcommand{\Lcof}{\ensuremath{\Lc^{r}\of{}^s_{p,q}}}
\newcommand{\Lca}{\ensuremath{\Lc^{r}\!a^s_{p,q}}}
\newcommand{\Lcb}{\ensuremath{\Lc^{r}\!b^s_{p,q}}}
\newcommand{\Lcf}{\ensuremath{\Lc^{r}\!f^s_{p,q}}}
\newcommand{\id}{\ensuremath{\mathrm{id}}}
\newcommand{\tr}{\ensuremath{\mathrm{tr\,}}}
\newcommand{\trd}{\ensuremath{\mathrm{tr}_d}}
\newcommand{\trL}{\ensuremath{\mathrm{tr}_L}}
\newcommand{\ext}{\ensuremath{\mathrm{ext}}}
\newcommand{\re}{\ensuremath{\mathrm{re\,}}}
\newcommand{\Rea}{\ensuremath{\mathrm{Re\,}}}
\newcommand{\Ima}{\ensuremath{\mathrm{Im\,}}}
\newcommand{\loc}{\ensuremath{\mathrm{loc}}}
\newcommand{\rloc}{\ensuremath{\mathrm{rloc}}}
\newcommand{\osc}{\ensuremath{\mathrm{osc}}}
\newcommand{\pr}{\pageref}
\newcommand{\wh}{\ensuremath{\widehat}}
\newcommand{\wt}{\ensuremath{\widetilde}}
\newcommand{\ol}{\ensuremath{\overline}}
\newcommand{\os}{\ensuremath{\overset}}
\newcommand{\Li}{\ensuremath{\overset{\circ}{L}}}
\newcommand{\Ai}{\ensuremath{\os{\, \circ}{A}}}
\newcommand{\Ci}{\ensuremath{\os{\circ}{\Cc}}}
\newcommand{\dom}{\ensuremath{\mathrm{dom \,}}}
\newcommand{\SA}{\ensuremath{S^r_{p,q} A}}
\newcommand{\SB}{\ensuremath{S^r_{p,q} B}}
\newcommand{\SF}{\ensuremath{S^r_{p,q} F}}
\newcommand{\Hc}{\ensuremath{\mathcal H}}
\newcommand{\Nc}{\ensuremath{\mathcal N}}
\newcommand{\Lci}{\ensuremath{\overset{\circ}{\Lc}}}
\newcommand{\bmo}{\ensuremath{\mathrm{bmo}}}
\newcommand{\BMO}{\ensuremath{\mathrm{BMO}}}
\newcommand{\cm}{\\[0.1cm]}
\newcommand{\Aa}{\ensuremath{\os{\, \ast}{A}}}
\newcommand{\Ba}{\ensuremath{\os{\, \ast}{B}}}
\newcommand{\Fa}{\ensuremath{\os{\, \ast}{F}}}
\newcommand{\Ha}{\ensuremath{\os{\, \ast}{H}}}
\newcommand{\Aas}{\ensuremath{\Aa{}^s_{p,q}}}
\newcommand{\Bas}{\ensuremath{\Ba{}^s_{p,q}}}
\newcommand{\Fas}{\ensuremath{\Fa{}^s_{p,q}}}
\newcommand{\Ca}{\ensuremath{\os{\, \ast}{{\mathcal C}}}}
\newcommand{\Cas}{\ensuremath{\Ca{}^s}}
\newcommand{\Car}{\ensuremath{\Ca{}^r}}
\newcommand{\bl}{$\blacksquare$}

\begin{center}
{\Large A note on Hausdorff--Young inequalities \\ in function spaces}
\\[1cm]
{Hans Triebel}
\\[0.2cm]
Institut f\"{u}r Mathematik\\
Friedrich--Schiller--Universit\"{a}t Jena\\
07737 Jena, Germany
\\[0.1cm]
email: hans.triebel@uni-jena.de
\end{center}

\begin{abstract}
\noindent This complements \cite{HST22}, \cite{HST23} and \cite{T22}. We use the same notation as there.
\end{abstract}

{\bfseries Keywords:} Fourier transform, weighted spaces, Hausdorff--Young inequalities

{\bfseries 2020 MSC:} 46E35

\section{Introduction}   \label{S1}
This note is not a paper or draft but a sketchy complement to \cite{HST22}, \cite{HST23} and \cite{T22}
 where we dealt with mapping properties of the Fourier transform $F$,
\begin{\eq}   \label{1.1}
(Ff)(\xi) = (2\pi)^{-n/2} \int_{\rn} e^{-ix \xi} f(x) \, \di x, \qquad f\in S'(\rn), \quad \xi \in \rn,
\end{\eq}
appropriately interpreted, between distinguished function spaces on \rn. We use the same notation as there. In particular,
\begin{\eq}   \label{1.2}
B^s_p (\rn) = B^s_{p,p} (\rn), \qquad s\in \real, \quad 1 \le p \le \infty,
\end{\eq}
are distinguished Besov spaces and
\begin{\eq}    \label{1.3}
H^s_p (\rn) = F^s_{p,2} (\rn), \qquad s\in \rn, \quad 1<p<\infty,
\end{\eq}
are their Sobolev counterparts. Let $w_\alpha (x) = (1+ |x|^2)^{\alpha/2}$, $x\in \rn$, $\alpha \in \real$, be the so--called admissible
weights. Then $A^s_p (\rn, w_\alpha)$, $A \in \{B,H \}$, $1\le p \le \infty$ (with $1<p<\infty$ for $H$--spaces), $s\in \real$, 
$\alpha \in \real$, are the weighted generalizations of the above spaces, normed by
\begin{\eq}   \label{1.4}
\| f \, | A^s_p (\rn, w_\alpha) \| = \| w_\alpha f \, |A^s_p (\rn) \|.
\end{\eq}
Let
\begin{\eq}     \label{1.5}
I_\alpha: \quad f \mapsto (w_\alpha \wh{f} )^\vee = (w_\alpha f^\vee )^\wedge, \qquad f\in S'(\rn), \quad \alpha \in \real,
\end{\eq}
be the well--known lifts in the above (unweighted) spaces $A^s_p (\rn)$, which can be extended to
\begin{\eq}   \label{1.6}
I_\alpha A^s_p (\rn, w_\beta) = A^{s- \alpha}_p (\rn, w_\beta),
\end{\eq}
\begin{\eq}   \label{1.7}
\| (w_\alpha \wh{f} )^\vee |A^{s- \alpha}_p (\rn, w_\beta ) \| \sim \|f \, | A^s_p (\rn, w_\beta) \|,
\end{\eq}
$\alpha \in \real$, $\beta \in \real$. Details about these weighted spaces and their indicated properties may be found in \cite[Theorem
6.5, pp.\,265--266]{T06} and have also been summarized in \cite[Remark 2.3]{HST22}. 

The mapping properties
\begin{\eq}  \label{1.8}
FS(\rn) = S(\rn), \qquad FS' (\rn) = S' (\rn)
\end{\eq}
and the classical Hausdorff--Young inequalities
\begin{\eq}   \label{1.9}
F: \quad L_{p'} (\rn) \hra L_p (\rn), \qquad 2 \le p \le \infty, \quad \frac{1}{p} + \frac{1}{p'} =1,
\end{\eq}
are cornerstones of Fourier analysis. In Section \ref{S2} we extend \eqref{1.9} to the above weighted spaces. Section \ref{S3} deals
with an application: We give a new proof of one of the crucial assertions about mapping properties of the Fourier transform as 
considered in \cite{HST23} and \cite{T22}.

\section{Hausdorff--Young inequalities}    \label{S2}
\subsection{Main assertions}   \label{S2.1}
We complement the isomorphic mapping \eqref{1.6}, \eqref{1.7},
\begin{\eq}          \label{2.1}
I_\gamma: \quad A^{s+\gamma}_p (\rn, w_\alpha) \hra A^s_p (\rn, w_\alpha)
\end{\eq}
by the isomorphic mapping
\begin{\eq}   \label{2.2}
W_\beta : \quad A^s_p (\rn, w_{\alpha + \beta} ) \hra A^s_p (\rn, w_\alpha),
\end{\eq}
where 
\begin{\eq}   \label{2.3}
W_\beta: \quad f \mapsto w_\beta f,
\end{\eq}
$A \in \{B,H \}$, $s\in \real$, $1\le p \le \infty$ ($1<p<\infty$ for $H$--spaces) and $\alpha, \beta, \gamma \in \real$. Furthermore
\begin{\eq}   \label{2.4}
F = W_{-\gamma} \circ I_{- \beta} \circ F \circ W_\beta \circ I_\gamma \qquad \text{in} \quad S'(\rn),
\end{\eq}
as has already been observed in \cite[(4.6)]{HST22}. By \eqref{1.9} and well--known embeddings one has
\begin{\eq}   \label{2.5}
F: \quad B^0_{p'} (\rn) \os{\id}{\hra} L_{p'} (\rn) \os{F}{\hra} L_p (\rn) \os{\id}{\hra} B^0_p (\rn),
\end{\eq}
$2\le p \le \infty$, $\frac{1}{p} + \frac{1}{p'}  =1$. 

\begin{theorem}   \label{T2.1}
{\em (i)} Let $2\le p \le \infty$, $\frac{1}{p} + \frac{1}{p'} =1$, $s\in \real$ and $\sigma \in \real$. Then $F$,
\begin{\eq}   \label{2.6}
F: \quad B^s_{p'} (\rn, w_\sigma) \hra B^\sigma_p (\rn, w_s),
\end{\eq}
is a continuous mapping.
\cm
{\em (ii)} Let $2\le p <\infty$, $\frac{1}{p} + \frac{1}{p'} =1$, $s\in \real$ and $\sigma \in \real$. Then $F$,
\begin{\eq}   \label{2.7}
F: \quad H^s_{p'} (\rn, w_\sigma) \hra H^\sigma_p (\rn, w_s),
\end{\eq}
is a continuous mapping.
\end{theorem}

\begin{proof}
We prove part (i) by reduction of \eqref{2.6} to \eqref{2.5}, based on \eqref{2.1}, \eqref{2.2},
\begin{\eq}   \label{2.8}
\begin{aligned}
I_s : \quad & B^s_{p'} (\rn, w_\sigma) && \hra B^0_{p'} (\rn, w_\sigma), \\
W_\sigma : \quad & B^0_{p'} (\rn, w_\sigma) && \hra B^0_{p'} (\rn), \\
F: \quad & B^0_{p'} (\rn) && \hra B^0_p (\rn), \\
I_{-\sigma}: \quad & B^0_p (\rn) && \hra B^\sigma_p (\rn), \\
W_{-s}: \quad & B^\sigma_p (\rn) && \hra B^\sigma_p (\rn, w_s).
\end{aligned}
\end{\eq}
Then \eqref{2.6} follows from \eqref{2.4}. The proof of \eqref{2.7} relies on \eqref{1.9} and $H^0_{p'} (\rn) = L_{p'} (\rn)$. 
\end{proof}

\subsection{Homogeneous modifications}  \label{S2.2}
If one replaces $f$ in \eqref{1.1} by $f(\lambda \cdot)$, $\lambda >0$, then one has
\begin{\eq}   \label{2.9}
\big( Ff (\lambda \cdot) \big) (\xi) = \lambda^{-n} (Ff)( \lambda^{-1} \xi), \qquad f\in S'(\rn), \quad \xi \in \rn,
\end{\eq}
appropriately interpreted. Inserted in the Hausdorff--Young inequality \eqref{1.9} one obtains
\begin{\eq}  \label{2.10}
\| \wh{f(\lambda \cdot)}(\xi) \,| L_{p} (\rn) \| = \lambda^{-n + \frac{n}{p}} \| \wh{f} \, | L_p (\rn) \| \le c \, \|f (\lambda \cdot)
\,| L_{p'} (\rn) \|,
\end{\eq}
$2 \le p \le \infty$, $\frac{1}{p} + \frac{1}{p'} =1$. This homogeneity does not play any role in Theorem \ref{T2.1}. But it suggests
to replace the inhomogeneous spaces $A^s_p (\rn)$, $A \in \{B,H \}$, $1\le p \le \infty$, $s\in \real$, by their homogeneous 
counterparts $\dot{A}^s_p (\rn)$. But to avoid the the usual bitter fighting modulo polynomials one can rely on the related tempered
homogeneous spaces $\Aa{}^s_p (\rn)$, $A\in \{B,H \}$, within the distinguished strip
\begin{\eq}   \label{2.11}
1 \le p \le \infty, \qquad - \frac{n}{p'} <s < \frac{n}{p}, \qquad \frac{1}{p} + \frac{1}{p'} =1,
\end{\eq}
in the framework of the dual pairing $\big( S(\rn), S'(\rn) \big)$, based on \cite{T15} and summarized partly in  \cite[Sections 4.1,
4.2]{T20}.

Let $w^\alpha (x) = |x|^\alpha$, $\alpha \in \real$, $x\in \rn$, be the homogeneous counterpart of the above weight $w_\alpha (x) =
(1 + |x|^2)^{\alpha/2}$ and let
\begin{\eq}   \label{2.12}
\| f \,| L_p (\rn, w^\alpha) \| = \| w^\alpha f \, | L_p (\rn) \|, \qquad 1 \le p \le \infty, \qquad \alpha \in \real,
\end{\eq}
in modification of  \cite[(2.169)--(2.171), p.\,37]{T15} where we already dealt with these weights. Let 
\begin{\eq}   \label{2.13}
- \frac{n}{p} < \alpha < \frac{n}{p'} \quad \text{where} \quad 1 \le  p \le \infty, \quad \frac{1}{p} + \frac{1}{p'} =1,
\end{\eq}
characterizing under which conditions $|x|^{\alpha p}$ belongs to the Muckenhoupt class $\Ac_p (\rn)$.  Then the first embedding in
\begin{\eq}   \label{2.14}
S(\rn) \hra L_p (\rn, w^\alpha) \hra S'(\rn)
\end{\eq}
follows from $\alpha > - \frac{n}{p}$. As for the second embedding we remark that one has for $f \in L_p (\rn, w^\alpha)$ and $\vp \in
S(\rn)$
\begin{\eq}   \label{2.15}
\Big| \int_{\rn} f(x)\,\vp(x) \, \di x \Big| \le \| f \, | L_p (\rn, w^\alpha ) \| \Big( \int_{\rn} |x|^{-\alpha p'} |\vp (x)|^{p'}
\di x \Big)^{1/p'}.
\end{\eq}
Then $\alpha p' <n$ justifies the right--hand side of \eqref{2.14}. In particular the duality
\begin{\eq}   \label{2.16}
L_p (\rn, w^\alpha)' = L_{p'} (\rn, w^{-\alpha} ), \qquad 1 \le p < \infty, \qquad \frac{1}{p} + \frac{1}{p'} =1,
\end{\eq}
can be interpreted  in the framework of the dual pairing $\big( S(\rn),S'(\rn) \big)$. 

Otherwise we rely on the notation and properties as described in \cite{T15}. In particular the well--known Gauss--Weierstrass
semi--group
\begin{\eq}   \label{2.17}
W_t w(x) = \frac{1}{(4\pi t)^{n/2}} \int_{\rn} e^{- \frac{|x-y|^2}{4t}} w(y) \, \di y = \frac{1}{(4\pi t)^{n/2}} \Big(w, e^{- \frac{|x-y|^2}{4t}} \Big), \qquad t>0,
\end{\eq}
$w\in S'(\rn)$, can be written on the Fourier side as
\begin{\eq}   \label{2.18}
\wh{W_t w}({\xi}) = e^{-t |\xi|^2} \wh{w} (\xi), \qquad \xi \in \rn.
\end{\eq}
Let (temporarily) 
\begin{\eq}   \label{2.19}
0<p \le \infty, \quad 0<q \le \infty \quad \text{and} \quad n \big( \frac{1}{p} - 1 \big) <s< \frac{n}{p}.
\end{\eq}
Then the related tempered homogeneous spaces $\Bas (\rn)$ admit the equivalent domestic quasi--norms
\begin{\eq}   \label{2.20}
\| f \, | \Bas (\rn) \|_m = \Big( \int^\infty_0 t^{(m- \frac{s}{2})q} \big\| \pa^m_t W_t f \, | L_p (\rn) \big\|^q \, \frac{\di t}{t}
\Big)^{1/q}
\end{\eq}
where $s/2 <m \in \no$ (usual modification if $q=\infty$). This is covered by \cite[Theorem 3.24, pp.\,79--80]{T15} and (including
$p=\infty$) \cite[Definition 4.6, pp.\,121--122]{T20}. There one finds also the necessary explanations and properties including the
counterpart
\begin{\eq}   \label{2.21}
S(\rn) \hra \Bas (\rn) \hra S' (\rn)
\end{\eq}
of \eqref{2.14}. Similarly for $\Fas (\rn)$, $p<\infty$, and in particular for the tempered homogeneous Sobolev spaces
\begin{\eq}   \label{2.22}
\Ha{}^s_p (\rn) = \Fa{}^s_{p,2} (\rn), \qquad 1<p<\infty.
\end{\eq}
If $p=\infty$ then one can choose $m=0$ and $\Ca{}^s (\rn) = \Ba{}^s_{\infty, \infty} (\rn)$, $-n <s<0$, can be introduced as the
collection of all $f\in S'(\rn)$ such that
\begin{\eq}   \label{2.23}
\| f \, |\Ca{}^s (\rn) \| = \sup_{t>0, x\in \rn} t^{-s/2} |W_t f(x) |
\end{\eq}
is finite. Let again $\Ba{}^s_{p,p} (\rn) = \Ba{}^s_p (\rn)$.

\begin{proposition}    \label{P2.2}
Let
\begin{\eq}   \label{2.24}
2 \le p \le \infty, \quad \frac{1}{p} + \frac{1}{p'} =1 \quad \text{and} \quad - \frac{n}{p'} <s< \frac{n}{p}.
\end{\eq}
Then
\begin{\eq}    \label{2.25}
F: \quad L_{p'} (\rn, w^s) \os{F}{\hra} \Ba{}^s_{p, p'} (\rn) \os{\id}{\hra} \Ba{}^s_p (\rn).
\end{\eq}
\end{proposition}

\begin{proof} {\em Step 1.} By \eqref{2.13}, \eqref{2.14} one has
\begin{\eq}    \label{2.26}
S(\rn) \hra L_{p'} (\rn, w^s) \hra S'(\rn)
\end{\eq}
and the homogeneity
\begin{\eq}   \label{2.27}
\| f(\lambda \cdot) \,| L_{p'} (\rn, w^s) \| = \lambda^{-s- \frac{n}{p'}} \| f\,| L_{p'} (\rn, w^s) \|,
\end{\eq}
$\lambda >0$. According to \cite[Theorem 3.24, (3.193), pp.\,79--80]{T15} the homogeneity exponent for $\Bas (\rn)$ is $s- \frac{n}{p}
$. Combined with \eqref{2.9} one has
\begin{\eq}   \label{2.28}
\| \wh{f(\lambda \cdot)} \, | \Ba{}^s_{p,p'} (\rn) \| = \lambda^{-n-s + \frac{n}{p}} \| f \,| \Ba{}^s_{p, p'} (\rn) \|,
\end{\eq}
$\lambda >0$. This shows that the claim \eqref{2.25} takes place within the dual pairing $\big( S(\rn), S'(\rn) \big)$ and that all
spaces have the same homogeneity. The last embedding in \eqref{2.25} follows from the monotonicity of $\Bas(\rn)$ with respect to $q$
(for fixed $s$ and $p$) which is the same as for the spaces $\Bs (\rn)$. This follows from the equivalent Fourier--analytical domestic
norms according to \cite[Theorems 3.3, 3.24, pp.\,48--49, 79--80]{T15}.
\cm
{\em Step2.} We prove the first mapping in \eqref{2.25}. Replacing $f$ in
\begin{\eq}   \label{2.29}
\pa^m_t W_t f (x) =  \big( (-1)^m |\xi|^{2m} e^{-t |\xi|^2} f^\vee \big)^\wedge  (x), \qquad x\in \rn,
\end{\eq}
by $\wh{f}$ then one obtains by \eqref{2.20} that
\begin{\eq}    \label{2.30}
\| \wh{f} \, | \Ba{}^s_{p,p'} (\rn) \|^{p'}_m = \int^\infty_0 t^{(m- \frac{s}{2})p'} \big\| \big( |\xi|^{2m} e^{-t|\xi|^2} f(\xi) 
\big)^\wedge \, |L_p (\rn) \big\|^{p'} \, \frac{\di t}{t}.
\end{\eq}
Application of the classical Hausdorff--Young inequality  \eqref{1.9} shows that
\begin{\eq}   \label{2.31}
\begin{aligned}
\| \wh{f} \, | \Ba{}^s_{p,p'} (\rn) \|^{p'}_m &\le c \int_{\rn} |\xi|^{2mp'} |f(\xi)|^{p'} \int^\infty_0 t^{(m- \frac{s}{2})p'}
e^{-t |\xi|^2 p'} \, \frac{\di t}{t} \, \di \xi \\
& = c \, \int_{\rn} |\xi|^{sp'} |f(\xi)|^{p'} \int^\infty_0 t^{(m- \frac{s}{2})p'} e^{-t p'} \frac{\di t}{t} \, \di \xi \\
&\sim \Gamma\big( p'(m - \frac{s}{2}) \big) \, \|f \, | L_{p'} (\rn, w^s) \|^{p'}
\end{aligned}
\end{\eq}
where $\Gamma(\tau)$ is the classical $\Gamma$--function. This proves the first mapping  in \eqref{2.25}.
\end{proof}

\begin{remark}   \label{R2.3}
The functional equation $z \Gamma (z) = \Gamma(z+1)$ for $z \in \comp$ with $\Rea z >0$ shows that one can replace $\Gamma \big( p'
(m - \frac{s}{2}) \big)$ in \eqref{2.31} by $\big( m - \frac{s}{2} \big)^{-1}$. Of special interest might be $s<0$ and $m=0$. Then
one has
\begin{\eq}   \label{2.32}
|s| \, \big\| \wh{f} \, | \Bas (\rn) \big\|^{p'}_0 \le c \| f \, | L_{p'} (\rn, w^s) \|^{p'}
\end{\eq}
where $c>0$ can be chosen independently of $s$ and $p' \le q \le \infty$. One may ask what happens if $s \uparrow 0$. Questions of this
type are quite fashionable nowadays. Does it converge to the classical Hausdorff--Young inequality? At least for $p=2$ one can say
the following. Let $- \frac{n}{2} <s<0$. Then one has by \eqref{2.19}, \eqref{2.20},
\begin{\eq}  \label{2.32a}
\| f \, | \Ha{}^s (\rn) \|^2_0 = \int^\infty_0 t^{-s} \|W_t f \,| L_2 (\rn) \|^2 \, \frac{\di t}{t}
\end{\eq}
where $\Ha{}^s (\rn) = \Ha{}^s_2 (\rn) = \Ba{}^s_2 (\rn) = \Ba{}^s_{2,2} (\rn)$. Inserting
\begin{\eq}    \label{2.32b}
\| W_t f \,| L_2 (\rn) \|^2 = \| (e^{-t |\xi|^2} \wh{f} )^\vee | L_2 (\rn) \|^2 = \| e^{t|\xi|^2} \wh{f} \, | L_2 (\rn) \|^2
\end{\eq}
one obtains
\begin{\eq}   \label{2.32c}
\begin{aligned}
\|f \, | \Ha{}^{s} (\rn) \|^2_0 = & \int_{\rn} |\wh{f} (\xi)|^2 \int^\infty_0 t^{-s} e^{-2t |\xi|^2} \, \frac{\di t}{t} \, \di \xi \\
&= \int_{\rn} 2^s |\xi|^{2s} |\wh{f} (\xi)|^2 \int^\infty_0 \tau^{-s} e^{-\tau} \frac{\di \tau}{\tau} \, \di \xi \\
& = 2^s \Gamma(|s|) \int_{\rn} |\xi|^{2s} |\wh{f} (\xi)|^2 \, \di \xi.
\end{aligned}
\end{\eq}
Using again $\Gamma (|s|) = |s|^{-1} (1 + o(|s|))$ with $o(|s|) \to 0$ if $|s| \to 0$ then it follows from Lebesgue's bounded 
convergence theorem that one has at least for $f\in S(\rn)$
\begin{\eq}   \label{2.32d}
\lim_{s \uparrow 0} \sqrt{|s|} \| f \, | \Ha{}^s (\rn) \|_0 = \| f \, | L_2 (\rn) \|.
\end{\eq}
The extension of \eqref{2.32} for $q=p'$ to $p' \le q \le \infty$ is not totally obvious because of the claimed independence of $c $ of
$s$ with $-\frac{n}{p'} <s<0$. But this follows from the possibility to discretize \eqref{2.20} with $m=0$ for $1<p \le \infty$, $1\le
q \le \infty$ and $-\frac{n}{p'}<s<0$ by
\begin{\eq}   \label{2.32e}
\begin{aligned}
\| f \, | \Bas (\rn) \|_0 &= \Big( \int^\infty_0 t^{\frac{|s|}{2}q} \| W_t f \, | L_p (\rn) \|^q \, \frac{\di t}{t} \Big)^{1/q} \\
&\sim \Big( \sum_{j\in \ganz} 2^{j \frac{|s|}{2}q} \| W_{2^j} f\, | L_p (\rn) \|^q \Big)^{1/q}
\end{aligned}
\end{\eq}
where the equivalence constants  can be chosen independently  of $s$, and the monotonicity of the $\ell_q$--spaces. Here \eqref{2.32e}
relies on
\begin{\eq}   \label{2.32f}
\| W_{2^{j+1}} f \, | L_p (\rn) \| \le \| W_t f \, | L_p (\rn) \| \le \| W_{2^j} f \, | L_p (\rn) \|
\end{\eq}
for $2^j \le t \le 2^{j+1}$, $j \in \ganz$. Finally we mention that \eqref{2.25} can be complemented by
\begin{\eq}   \label{2.32g}
F: \quad L_{p'} (\rn, w^s) \hra \Bas (\rn)
\end{\eq}
if, and only if, $q \ge p'$. If $q \ge p'$ then \eqref{2.32g} follows from  \eqref{2.25} and the monotonicity  of the spaces $\Bas 
(\rn)$ with respect to $q$ which is the same as for the inhomogeneous spaces $\Bs (\rn)$. Let $\psi$ be a non--trivial $C^\infty$
function with, say, $\supp \psi \subset \{ x: |x| <1 \}$. Let
\begin{\eq}   \label{2.32h}
\psi_j (x) = \psi (x-x^j) \quad \text{ such that} \quad \supp \psi_j \subset \{x: 2^j  <|x| < 2^{j+1} \},
\end{\eq}
$j\in \nat$. Let $\lambda = \{ \lambda_j \}_{j\in \nat} \subset \comp$ and
\begin{\eq}  \label{2.32i}
f_\lambda = \sum_{j\in \nat}  \lambda_j \psi_j.
\end{\eq}
Then
\begin{\eq}   \label{2.32j}
\| f_\lambda \, | L_{p'} (\rn, w^s) \| \sim \Big( \sum_{j\in \nat} 2^{jsp'} |\lambda_j |^{p'} \Big)^{1/p'}.
\end{\eq}
Inserting $\wh{f_\lambda}$ in the equivalent domestic Fourier--analytic norm for $\Bas (\rn)$ with $1\le q <\infty$ one obtains
\begin{\eq}   \label{2.32k}
\| \wh{f_\lambda} \, | \Bas (\rn) \| \sim \Big(\sum_{j\in \nat} 2^{jsq} |\lambda_j |^q \| (\vp_j  \psi_j )^\wedge \, | L_p (\rn)
\|^q \Big)^{1/q}.
\end{\eq}
We may assume that $\vp_j \psi_j = \psi_j$. With $\wh{\psi_j} (\xi) = e^{-i \xi x^j} \psi(\xi)$ it follows from \eqref{2.32k} that
\begin{\eq}   \label{2.32l}
\| \wh{f_\lambda} \, | \Bas (\rn) \| \sim \Big( \sum_{j\in \nat} 2^{jsq} |\lambda_j |^q \Big)^{1/q}.
\end{\eq}
Then \eqref{2.32g} requires $q \ge p'$. 
\end{remark}

One can extend the above assertions to the Sobolev spaces according to \eqref{2.22}. Furthermore one can apply duality to change the
role of source spaces and target spaces in \eqref{2.25}, excluding now $p= \infty$.

\begin{theorem}   \label{T2.4}
Let
\begin{\eq}    \label{2.33}
2 \le p <\infty, \quad \frac{1}{p} + \frac{1}{p'} =1 \quad \text{and} \quad - \frac{n}{p'} <s< \frac{n}{p}.
\end{\eq}
Then  
\begin{\eq}   \label{2.34}
F: \quad L_{p'} (\rn, w^s) \os{F}{\hra} \Ba{}^s_{p,p'} (\rn) \os{\id}{\hra} \Ha{}^s_p  (\rn) \os{\id}{\hra} \Ba{}^s_p (\rn)
\end{\eq}
and
\begin{\eq}    \label{2.35}
F: \quad \Ba{}^{-s}_{p'} (\rn) \os{\id}{\hra} \Ha{}^{-s}_{p'} (\rn) \os{\id}{\hra} \Ba{}^{-s}_{p',p} (\rn) \os{F}{\hra} L_p 
(\rn, w^{-s}).
\end{\eq}
\end{theorem}

\begin{proof}
The mappings in \eqref{2.34} follow from \eqref{2.25}, $p<\infty$, and the additional embeddings which are the same as for the related
inhomogeneous spaces according to \cite[(3.44), p.\,52]{T15}, extended to $s$ in \eqref{2.33}, and \eqref{2.22}. The duality
\begin{\eq}   \label{2.36}
\Ba{}^s_p (\rn)' = \Ba{}^{-s}_{p'} (\rn) \quad \text{and} \quad \Ha{}^s_p (\rn)' = \Ha{}^{-s}_{p'} (\rn)
\end{\eq}
for
\begin{\eq}   \label{2.37}
1<p<\infty, \quad \frac{1}{p} + \frac{1}{p'} =1 \quad \text{and} \quad - \frac{n}{p'} <s< \frac{n}{p}
\end{\eq}
within the dual pairing $\big( S(\rn), S'(\rn) \big)$ can be proved in the same way as for the corresponding inhomogeneous spaces in
\cite[Theorem 2.11.2, p.\,178]{T83} based on related domestic Fourier--analytical norms according to 
\cite[Theorem 3.24, pp.\,79--80]{T15}. Then \eqref{2.35} follows from \eqref{2.34}, \eqref{2.36} and \eqref{2.16}. The restrictions
for $s$ in \eqref{2.33} and $- \frac{n}{p} < -s < \frac{n}{p'}$ ensure that all $B$--spaces and all $H$--spaces in \eqref{2.34} and
\eqref{2.35} are within the distinguished strip.
\end{proof}

\begin{remark}   \label{R2.5}
Assertions  of type \eqref{2.35} in the context of homogeneous spaces have also been treated recently in \cite{Bou22}, called Szasz
theorems with a reference to \cite{Pee76} as far as this notation is concerned.
\end{remark}

\begin{problem}    \label{P2.6}
One may compare Theorem \ref{T2.1} for the inhomogeneous weighted spaces $A^s_p (\rn, w_\sigma)$, $A \in \{B,H \}$, with their 
homogeneous counterparts in Proposition \ref{P2.2} and Theorem \ref{T2.4}. So far we restricted there weights to the related $L_p$--spaces. However some extensions seem to be possible. For this purpose one may use the lifts as discussed in \cite[Section 3.14, 
pp.\,97--98]{T15}. In addition we already introduced in \cite[Section 2.6.4, pp.\,36--41]{T15} weighted spaces $\Aas (\rn, w^\alpha)$
and described some properties. On this basis it seems to be possible to elaborate the above theory. In particular one may ask for a
homogeneous weighted counterpart of \eqref{2.4}.
\end{problem}

\section{An application}      \label{S3}
In modification of \cite{T22} we dealt in \cite{HST23} with mapping properties of
\begin{\eq}   \label{3.1}
F: \quad SA^{s_1}_p (\rn) \hra TA^{s_2}_p (\rn)
\end{\eq}
where  $A \in \{B,H \}$, $1 \le p \le \infty$ ($1<p<\infty$ for $H$--spaces),
\begin{\eq}   \label{3.2}
SA^s_p (\rn) = A^{s^n_p +s}_p (\rn) \quad \text{and} \quad TA^s_p (\rn) = A^{t^n_p -s}_p (\rn)
\end{\eq}
with $s\in \real$,
\begin{\eq}    \label{3.3}
d^n_p = 2n \big( \frac{1}{p} - \frac{1}{2} \big) \quad \text{and} \quad s^n_p = \max (0, d^n_p), \quad t^n_p = \min (0, d^n_p).
\end{\eq}

\begin{theorem}   \label{T3.1}
Let $A \in \{B,H \}$, $1 \le p \le \infty$ $(1<p<\infty$ for $H$--spaces$)$ and $s_1 \in \real$, $s_2 \in \real$. Then there is a
continuous mapping
\begin{\eq}   \label{3.4}
F: \quad SA^{s_1}_p (\rn) \hra TA^{s_2}_p (\rn)
\end{\eq}
if, and only if, both $s_1 \ge 0$ and $s_2 \ge 0$. This mapping is compact if, and only if, both $s_1 >0$ and $s_2 >0$.
\end{theorem}

\begin{remark}   \label{R3.2}
This coincides with \cite[Theorems 3.1 and 3.2]{HST23}. Here either $A=B$ on both sides of \eqref{3.4} or $A=H$ on both sides of
\eqref{3.4}.
\end{remark}

The proof in \cite{T22} that the continuity of the mapping
\begin{\eq}   \label{3.5}
F: \quad B^{s_1}_p (\rn) \hra B^{d^n_p -s_2}_p (\rn), \qquad 2 \le p \le \infty,
\end{\eq}
requires both $s_1 \ge 0$ and $s_2 \ge 0$ relies on preceding compactness assertions and interpolation. The $H$--spaces can be 
incorporated afterwards, whereas the related counterpart  for $1\le p \le 2$ ($1<p \le 2$ for $H$--spaces) is a matter of duality
as detailed in \cite{T22} and \cite{HST23}. We give a new proof of this claim which relies on the extended Hausdorff--Young inequality
according to Theorem \ref{T2.1}. It is again sufficient to concentrate on \eqref{3.5}.

\begin{proposition}   \label{[P3.3}
Let $A \in \{B, H \}$, $1 \le p \le \infty$ $(1 <p<\infty$ for $H$--spaces$)$ and $s_1 \in \real$, $s_2 \in \real$. If there is a 
continuous mapping
\begin{\eq}   \label{3.6}
F: \quad SA^{s_1}_p (\rn) \hra TA^{s_2}_p (\rn)
\end{\eq}
then both $s_1 \ge 0$ and $s_2 \ge 0$.
\end{proposition}

\begin{proof}
As already said it is sufficient to prove that there is no continuous mapping
\begin{\eq}   \label{3.7}
F: \quad B^{s_1}_p (\rn) \hra B^{d^n_p -s_2}_p(\rn), \qquad 2 \le p \le \infty, \quad d^n_p = \frac{2n}{p} -n ,
\end{\eq}
if either $s_1 <0$ or $s_2 <0$. Let $s_1 \in \real$, $s_2 \in \real$ and $\frac{1}{p} + \frac{1}{p'} =1$. By \eqref{2.6} one has
\begin{\eq}   \label{3.8}
F^{-1}: \quad B^0_{p'} (\rn, w_{s_1} ) \hra B^{s_1}_p (\rn).
\end{\eq}
Combined with the assumed mapping \eqref{3.7} it follows that there would be a continuous mapping 
\begin{\eq}   \label{3.9}
\id: \quad B^0_{p'} (\rn, w_{s_1} ) \hra B^{d^n_p - s_2}_p (\rn).
\end{\eq}
This requires according to \cite[Theorem 6.7, p.\,266]{T06} that $s_1 \ge 0$ (for any $s_2 \in \real$). If $p=2$ then one has by 
\eqref{3.9} that $s_2 \ge 0$ (for any $s_1 \in \real$). If $2<p \le \infty$ then one may assume that $d^n_p - s_2 <0$. Now it follows
again from \cite[Theorem 6.7, p.\,266]{T06} that
\begin{\eq}   \label{3.10}
- \frac{n}{p'} \ge \frac{2n}{p} - n -s_2 - \frac{n}{p}, \quad \text{which means $s_2 \ge 0$}
\end{\eq}
(for any $s_1 \in \real$).
\end{proof}

\end{document}